\documentclass[a4paper]{amsart}
\usepackage{graphicx}
\usepackage{amsmath}
\usepackage{amssymb}
\usepackage{oldgerm}
\usepackage[all]{xy}

\newcommand{\Hom}{\operatorname{Hom}\nolimits}
\renewcommand{\Im}{\operatorname{Im}\nolimits}

\newcommand{\Ker}{\operatorname{Ker}\nolimits}

\newcommand{\Ann}{\operatorname{Ann}\nolimits}

\newcommand{\Ext}{\operatorname{Ext}\nolimits}

\newcommand{\Maxspec}{\operatorname{MaxSpec}\nolimits}

\newcommand{\HH}{\operatorname{HH}\nolimits}

\newcommand{\m}{\operatorname{\mathfrak{m}}\nolimits}

\newcommand{\ra}{\operatorname{\mathfrak{r}}\nolimits}
\newcommand{\La}{\operatorname{\Lambda}\nolimits}

\newcommand{\op}{\operatorname{op}\nolimits}
\newcommand{\V}{\operatorname{V}\nolimits}

\newcommand{\e}{\operatorname{e}\nolimits}

\newcommand{\Char}{\operatorname{char}\nolimits}

\newtheorem{theorem}{Theorem}[section]

\newtheorem{corollary}[theorem]{Corollary}

\newtheorem{lemma}[theorem]{Lemma}
\newtheorem{proposition}[theorem]{Proposition}
\theoremstyle{definition}

\theoremstyle{definition}

\theoremstyle{definition}

\theoremstyle{remark}

\theoremstyle{remark}

\theoremstyle{definition}

\begin{document}
\title{The Avrunin-Scott theorem for quantum complete intersections}
\author{Petter Andreas Bergh \& Karin Erdmann}
\address{Petter Andreas Bergh \\ Institutt for matematiske fag \\
NTNU \\ N-7491 Trondheim \\ Norway} \email{bergh@math.ntnu.no}
\address{Karin Erdmann \\ Mathematical Institute \\ 24-29 St.\ Giles \\ Oxford OX1 3LB \\ United Kingdom}
\email{erdmann@maths.ox.ac.uk}

\thanks{The first author was supported by NFR Storforsk grant no.\
167130}



\maketitle

\begin{abstract}
We prove the Avrunin-Scott theorem for quantum complete intersections; the rank variety of a module is isomorphic to its support variety.
\end{abstract}

\section{Introduction}\label{sec1}

Inspired by the impact of the theories of varieties for modules over group algebras, similar theories have been studied for other classes of algebras. For example, using Hochschild cohomology, Snashall and Solberg developed a theory of support varieties for finite dimensional algebras in \cite{Snashall}. As shown in \cite{Erdmann}, this theory is very powerful when the cohomology of the algebra satisfies sufficient finite generation, and for selfinjective such algebras the theory shares many of the properties of that for group algebras. However, support varieties are difficult to compute. In \cite{Carlson}, Carlson introduced rank varieties for modules over group algebras of elementary abelian groups, varieties defined without using cohomology. Given a module over such an algebra, its rank variety is very explicit and easy to compute. Moreover, Avrunin and Scott proved in \cite{AvruninScott} that the support variety of a module is in fact isomorphic to its rank variety. Needless to say, this result has had important consequences (see, for example, the introduction
in \cite{AvruninScott} or \cite{ErdmannHolloway}).

Motivated by this, the second author and Holloway introduced in \cite{ErdmannHolloway} rank varieties for truncated polynomial algebras in which the generators square to zero. Such algebras also have support varieties, and it was shown that these two varieties are isomorphic.

In this paper, we study rank varieties for quantum complete intersections, a class of algebras originating from work by Manin and Avramov, Gasharov and Peeva (cf.\ \cite{Manin} and \cite{Avramov}). When the defining parameters are roots of unity, then these algebras have support varieties; it was shown in \cite{BerghOppermann} that finite generation of cohomology holds. However, certain quantum complete intersections also have rank varieties, and one would therefore like to know if and how the support and rank varieties are related. We show that they are indeed isomorphic.

\section{Varieties}\label{sec2}

Throughout this section, let $k$ be an algebraically closed field. All modules considered are assumed to be left modules and finitely generated. We start by recalling the definitions and some results on support varieties; details can be found in \cite{Erdmann} and \cite{Snashall}. 

Let $\La$ be a finite dimensional $k$-algebra with Jacobson radical $\ra$. We denote by $\La^{\e}$ the enveloping algebra $\La \otimes_k \La^{\op}$ of $\La$. The $n$th \emph{Hochschild cohomology} group of $\La$, denoted $\HH^n ( \La )$, is the vector space $\Ext_{\La^{\e}}^n ( \La, \La )$ of $n$-fold bimodule extensions of $\La$ with itself. The Yoneda product turns $\HH^* ( \La ) = \oplus_{n=0}^{\infty} \HH^n ( \La )$ into a graded $k$-algebra, the Hochschild cohomology ring of $\La$. This algebra is graded commutative, that is, given two homogeneous elements $\eta, \theta \in \HH^* ( \La )$, the equality $\eta \theta = (-1)^{|\eta||\theta|} \theta \eta$ holds. In particular, the even part $\HH^{2*}( \La ) = \oplus_{n=0}^{\infty} \HH^{2n} ( \La )$ of $\HH^* ( \La )$ is a commutative ring.

Every homogeneous element $\eta \in \HH^* ( \La )$ can be represented by a bimodule extension
$$\eta \colon 0 \to \La \to B_{|\eta|} \to \cdots \to B_1 \to \La \to 0.$$
Given a $\La$-module $M$, the complex
$$\eta \otimes_{\La} M \colon 0 \to M \to B_{|\eta|} \otimes_{\La} M \to \cdots \to B_{|\eta|} \otimes_{\La} M \to M \to 0$$
is exact, hence the tensor product $- \otimes_{\La} M$ induces a homomorphism
$$\HH^* ( \La ) \xrightarrow{\varphi_M} \Ext_{\La}^* (M,M)$$
of graded $k$-algebras. Thus, given two $\La$-modules $M$ and $N$, the graded vector space $\Ext_{\La}^*(M,N)$ becomes a graded module over $\HH^*( \La )$ in two ways; either via $\varphi_M$ or via $\varphi_N$. These two module structures are equal up to a graded sign, that is, given homogeneous elements $\eta \in \HH^* ( \La )$ and $\theta \in \Ext_{\La}^*(M,N)$, the equality $\varphi_N ( \eta ) \circ \theta = (-1)^{|\eta||\theta|} \theta \circ \varphi_M ( \eta )$ holds, where ``$\circ$" denotes the Yoneda product.

Let $H$ be a commutative graded subalgebra of $\HH^* ( \La )$. The \emph{support variety} of an ordered pair $(M,N)$ of $\La$-modules (with respect to $H$), denoted $\V_H(M,N)$, is defined as
$$\V_H(M,N) \stackrel{\text{def}}{=} \{ \m \in \Maxspec H \mid \Ann_H \Ext_{\La}^*(M,N) \subseteq \m \},$$
where $\Maxspec H$ is the set of maximal ideals of $H$. There are equalities 
$$\V_H(M, \La / \ra ) = \V_H(M,M) = \V_H( \La / \ra, M),$$
and this set is defined to be the support variety $V_H(M)$ of $M$. 

In general, support varieties do not contain any homological information on the modules involved. For example, if $\HH^*( \La )$ is finite dimensional, that is, if $\HH^n ( \La )=0$ for $n \gg 0$, then the support variety of any pair of modules is trivial. However, under certain finiteness conditions, the situation is quite different. Suppose  $H$ is Noetherian and $\Ext_{\La}^*(M,N)$ is a finitely generated $H$-module for all $\La$-modules $M$ and $N$ (this is equivalent to $\Ext_{\La}^*( \La / \ra, \La / \ra)$ being finitely generated over $H$). In this case, the dimension of $\V_H(M,N)$ equals the polynomial rate of growth of $\Ext_{\La}^*(M,N)$. In particular, the dimension of the support variety of a module equals its complexity, hence a module
has finite projective dimension if and only if its support variety is trivial. Moreover, the dimension of the support variety of a module is one if and only if its minimal projective resolution is bounded. Under the finite generation hypothesis given, this happens precisely when the module is eventually periodic, that is, when its minimal projective resolution becomes periodic from some step on. 

When the algebra $\La$, in addition to satisfying the finite generation hypothesis, is also selfinjective, then the projective support variety of an indecomposable module is connected. Namely, let $M$ be a $\La$-module whose support variety decomposes as $\V_H(M) = V_1 \cup V_2$, where $V_1$ and $V_2$ are closed homogeneous varieties such that $\V_1 \cap \V_2$ is trivial. Then there are submodules $M_1$ and $M_2$ of $M$, with the property that $M = M_1 \oplus M_2$ and $\V_H(M_i) =V_i$. In particular, this implies that the support variety of an indecomposable periodic module is a single line.

We now turn to the class of algebras for which we will prove the Avrunin-Scott theorem. These are analogues of truncated polynomial rings. Let $c \ge 1$ be an integer, and let ${\bf{q}} = (q_{ij})$ be a $c \times c$ commutation matrix with entries in $k$. That is, the diagonal entries $q_{ii}$ are all $1$, and $q_{ij}q_{ji}=1$ for $i \neq j$. Furthermore, let ${\bf{a}}_c = (a_1, \dots, a_c)$ be an ordered sequence of $c$ integers with $a_i \ge 2$. The \emph{quantum complete intersection} $A_{\bf{q}}^{{\bf{a}}_c}$ determined by these data is the algebra
$$A_{\bf{q}}^{{\bf{a}}_c} \stackrel{\text{def}}{=} k \langle x_1, \dots, x_c \rangle / (x_i^{a_i}, x_ix_j-q_{ij}x_jx_i),$$
which is selfinjective and finite dimensional of dimension $\prod a_i$. It was proved in \cite{BerghOppermann} that if all the commutators $q_{ij}$ are roots of unity, then $\HH^{2*}( A_{\bf{q}}^{{\bf{a}}_c} )$ is Noetherian, and $\Ext_{A_{\bf{q}}^{{\bf{a}}_c}}^*(M,N)$ is a finitely generated $\HH^{2*}( A_{\bf{q}}^{{\bf{a}}_c} )$-module for all $A_{\bf{q}}^{{\bf{a}}_c}$-modules $M$ and $N$. Thus, in this case, the support varieties with respect to $\HH^{2*} (A_{\bf{q}}^{{\bf{a}}_c})$ detect projective and periodic modules, as we saw above. The Krull dimension of $\HH^{2*}( A_{\bf{q}}^{{\bf{a}}_c} )$ is $c$, that is, the number of generators defining the quantum complete intersection. Therefore, the support varieties are homogeneous affine subsets of $k^c$.

Now fix an integer $a \ge 2$, and define $a'$ by
$$a' \stackrel{\text{def}}{=} \left \{ \begin{array}{ll}
a/ \gcd ( a, \Char k ) & \text{if } \Char k >0 \\
a & \text{if } \Char k =0.
\end{array} \right.$$
Let $q \in k$ be a primitive $a'$th root of unity, let ${\bf{q}}$ be the commutation matrix with $q_{ij}=q$ for $i<j$, and let ${\bf{a}}_c$ be the $c$-tuple $(a, \dots, a)$. Then we denote the corresponding quantum complete intersection $A_{\bf{q}}^{{\bf{a}}_c}$ by $A^c_q$, i.e.\
$$A^c_q = k \langle x_1, \dots, x_c \rangle / ( \{ x_i^a \}, \{ x_ix_j-qx_jx_i \}_{i<j} ).$$
Given any $c$-tuple $\lambda = ( \lambda_1, \dots, \lambda_c ) \in k^c$, denote the element $\lambda_1 x_1 + \cdots + \lambda_c x_c \in A^c_q$ by $u_{\lambda}$, and let $k[ u_{\lambda} ]$ be the subalgebra of $A^c_q$ generated by this element. Then $u_{\lambda}^a =0$ by \cite[Lemma 2.3]{Benson}, and the \emph{rank variety} of an $A^c_q$-module $M$, denoted $\V^r_{A^c_q}(M)$, is defined as
$$\V^r_{A^c_q}(M) \stackrel{\text{def}}{=} \{ 0 \} \cup \{ 0 \neq \lambda \in k^c \mid M \text{ is not a projective $k[ u_{\lambda} ]$-module} \}.$$ 
When $\lambda$ is nonzero, then since $u_{\lambda}^a =0$, the subalgebra $k[ u_{\lambda} ]$ is isomorphic to the truncated polynomial ring $k[x]/(x^a)$. Therefore, the requirement that an $A^c_q$-module $M$ is not $k[ u_{\lambda} ]$-projective is equivalent to the requirement that the $k$-linear map $M \xrightarrow{\cdot u_{\lambda}} M$ satisfies
$$\dim \Im ( \cdot u_{\lambda} ) < ((n-1)/n) \dim M.$$
This explains the choice of terminology and shows that, like support varieties, rank varieties are homogeneous affine subsets of $k^c$.

Our aim is to show that the rank variety $\V^r_{A^c_q}(M)$ of an $A^c_q$-module $M$ is isomorphic to its support variety $\V_{\HH^{2*}(A^c_q)}(M)$. We do this by exploiting the structure of the $\Ext$-algebra of the simple $A^c_q$-module $k$. These algebras were determined for all quantum complete intersections in \cite[Theorem 5.3]{BerghOppermann}. For $A^c_q$, it is given by
$$\Ext_{A^c_q}^*(k,k) = k \langle y_1, \dots, y_c,z_1, \dots, z_c \rangle / I,$$
where $I$ is the ideal generated by the relations
$$\left (
\begin{array}{ll}
y_iy_j + q_{ij} y_jy_i & \text{for all } i \neq j \\
y_iz_j - z_jy_i & \text{for all } i,j \\
z_iz_j - z_jz_i & \text{for all } i,j \\
y_i^2-z_i & \text{for all $i$ if } a=2 \\
y_i^2 & \text{for all $i$ if } a \neq 2 
\end{array}
\right ) $$
and where $|y_i| =1$ and $|z_i|=2$. We see that the commutative subalgebra of $\Ext_{A^c_q}^*(k,k)$ generated by $z_1, \dots, z_c$ is the polynomial ring $k[z_1, \dots, z_c]$, and this subalgebra contains the image of the ring homomorphism $\HH^{2*} (A^c_q) \xrightarrow{\varphi_k} \Ext_{A^c_q}^*(k,k)$. By \cite{Oppermann} this image is actually the whole ring $k[z_1, \dots, z_c]$, hence there exists a polynomial subalgebra $H = k[ \eta_1, \dots, \eta_c]$ of $\HH^{2*} (A^c_q)$, with the property that $|\eta_i| =2$ and $\varphi_k ( \eta_i ) =z_i$ for all $i$. 

Let $H$ be such a polynomial subalgebra of $\HH^{2*} (A^c_q)$. Then for all $A^c_q$-modules $M$, the support varieties $\V_H (M)$ and $\V_{\HH^{2*} (A^c_q)}(M)$ are isomorphic. To see this, denote by $R$ the polynomial subalgebra of $\Ext_A^*(k,k)$ generated by the $z_i$. For each $A$-module $M$, the set
$$\{ \m \in \Maxspec R \mid \Ann_R \Ext_A^*(M,k) \subseteq \m \}$$ 
is an affine subset of $\Maxspec R$, namely the \emph{projective relative support variety} of $M$ with respect to $R$ (introduced in \cite{BerghSolberg}). Since $\varphi_k (H) = \varphi_k ( \HH^{2*}(A) ) =R$, it follows from \cite[Proposition 3.4]{BerghSolberg} that the above affine subset of $\Maxspec R$ is isomorphic to both $\V_H (M)$ and $\V_{\HH^{2*} (A)}(M)$. This shows that $\V_H (M)$ and $\V_{\HH^{2*} (A)}(M)$ are isomorphic.

\section{The Avrunin-Scott theorem}\label{sec3}

We keep the same notation as we used in the previous section. That is, throughout this section, we let $k$ be an algebraically closed field, and we assume that all modules are finitely generated and left modules. We fix an integer $a \ge 2$, and define $a'$ by 
$$a' \stackrel{\text{def}}{=} \left \{ \begin{array}{ll}
a/ \gcd ( a, \Char k ) & \text{if } \Char k >0 \\
a & \text{if } \Char k =0 .
\end{array} \right.$$
Moreover, we fix an integer $c \ge 1$ and a primitive $a'$th root of unity $q \in k$, and work with the corresponding quantum complete intersection $A_q^c$ given by
$$A^c_q = k \langle x_1, \dots, x_c \rangle / ( \{ x_i^a \}, \{ x_ix_j-qx_jx_i \}_{i<j} ).$$
To simplify notation, we denote this algebra by $A$. We fix a polynomial subalgebra $H=k[ \eta_1, \dots, \eta_c]$ of $\HH^{2*}(A)$, with $\eta_i \in \HH^{2}(A)$ and $\varphi_k ( \eta_i ) =z_i \in \Ext_A^2(k,k)$. Moreover, we identify the maximal ideals of $H$ with the points of $k^c$. Finally, given a nonzero point $\lambda = ( \lambda_1, \dots, \lambda_c) \in k^c$, we denote the corresponding line in $k^c$ by $\ell_{\lambda}$, and the element $\lambda_1x_1+ \cdots + \lambda_cx_c \in A$ by $u_{\lambda}$.

As we saw in the previous section, the Avrunin-Scott theorem, that is, establishing an isomorphism $\V^r_A(M) \simeq \V_{\HH^{2*}(A)}(M)$ for all $A$-modules $M$, is equivalent to showing that $\V^r_{A}(M)$ is isomorphic to the support variety $\V_H(M)$ with respect to $H$. We do this is in a number of steps, the first of which is the following result. It provides a stable map description of rank varieties.

\begin{proposition}\label{stablemap}
Given a nonzero point $\lambda \in k^c$, the implications
$$\ell_{\lambda} \subseteq \V^r_A (M) \Leftrightarrow \underline{\Hom}_A (Au_{\lambda},M) \neq 0 \Leftrightarrow \underline{\Hom}_A (Au_{\lambda}^{a-1},M) \neq 0$$
hold for every $A$-module $M$.
\end{proposition}

\begin{proof}
Applying $\Hom_A(-,M)$ to the right exact sequence
$$A \xrightarrow{\cdot u_{\lambda}^{a-1}} A \xrightarrow{\cdot u_{\lambda}} Au_{\lambda} \to 0$$
gives the left exact sequence
$$0 \to \Hom_A(Au_{\lambda},M) \to M \xrightarrow{\cdot u_{\lambda}^{a-1}} M$$
of vector spaces. From this sequence we obtain the isomorphism $\Hom_A(Au_{\lambda},M) \simeq \Ker ( M \xrightarrow{\cdot u_{\lambda}^{a-1}} M )$, and interchanging the roles of the two elements $u_{\lambda}$ and $u_{\lambda}^{a-1}$, we obtain the isomorphism $\Hom_A(Au_{\lambda}^{a-1},M) \simeq \Ker ( M \xrightarrow{\cdot u_{\lambda}} M )$. Next, consider the two exact sequences
$$0 \to Au_{\lambda}^{a-1} \hookrightarrow A \xrightarrow{\cdot u_{\lambda}} Au_{\lambda} \to 0$$
$$0 \to Au_{\lambda} \hookrightarrow A \xrightarrow{\cdot u_{\lambda}^{a-1}} Au_{\lambda}^{a-1} \to 0,$$
which show that the modules $Au_{\lambda}^{a-1}$ and $Au_{\lambda}$ are syzygies of each other. Applying $\Hom_A(-,M)$ to the first sequence, we obtain the long exact sequence
$$0 \to \Hom_A(Au_{\lambda},M) \to M \to \Hom_A(Au_{\lambda}^{a-1},M) \to \Ext_A^1(Au_{\lambda},M) \to 0$$
of vector spaces. The isomorphisms we obtained above, together with the isomorphisms $\Ext_A^1(Au_{\lambda},M) \simeq \underline{\Hom}_A ( \Omega_A^1(Au_{\lambda}),M) \simeq \underline{\Hom}_A (Au_{\lambda}^{a-1},M)$, give the equality
$$\dim \Ker ( \cdot u_{\lambda}^{a-1} ) + \dim \Ker ( \cdot u_{\lambda} ) = \dim M + \dim \underline{\Hom}_A (Au_{\lambda}^{a-1},M).$$
Similarly, by applying $\Hom_A(-,M)$ to the other short exact sequence, we obtain the equality
$$\dim \Ker ( \cdot u_{\lambda} ) + \dim \Ker ( \cdot u_{\lambda}^{a-1} ) = \dim M + \dim \underline{\Hom}_A (Au_{\lambda},M).$$
By \cite[Remarks 3.2(ii)]{ErdmannHolloway}, the requirement that $M$ is not a projective $k[u_{\lambda}]$-module is equivalent to the requirement
$$\dim \Ker ( \cdot u_{\lambda} ) + \dim \Ker ( \cdot u_{\lambda}^{a-1} ) > \dim M,$$
and so the result follows.
\end{proof}

Our next aim is to determine the support variety $\V_H ( Au_{\lambda} )$ for all nonzero points $\lambda \in k^c$. We start with the following general result. Recall that an algebra is \emph{Frobenius} if, as a left module over itself, it is isomorphic to its vector space dual.

\begin{lemma}\label{adjoint}
Let $\La$ and $\Gamma$ be two finite dimensional $k$-algebras, with $\La$ selfinjective and $\Gamma$ Frobenius. Furthermore, let $B$ be a $\Gamma$-$\La$-bimodule which is projective both as a $\Gamma$-module and as a $\La$-module.  Then for every $\La$-module $X$ and every $\Gamma$-module $Y$, there is a natural isomorphism
$$\Ext_{\La}^n (X, D(B) \otimes_{\Gamma} Y) \simeq \Ext_{\Gamma}^n(B \otimes_{\La} X, Y)$$
for each $n \ge 0$.
\end{lemma}

\begin{proof}
Adjointness gives a natural isomorphism
$$\Hom_{\La} (X, \Hom_{\Gamma}(B,Y)) \simeq \Hom_{\Gamma}(B \otimes_{\La} X, Y)$$
for every $\La$-module $X$ and every $\Gamma$-module $Y$. The functor $\Hom_{\Gamma}(B,-)$ is isomorphic to $\Hom_{\Gamma}(B, \Gamma \otimes_{\Gamma} -)$, which in turn is isomorphic to $\Hom_{\Gamma}(B, \Gamma ) \otimes_{\Gamma} -$ since $B$ is a projective $\Gamma$-module. Now $\Gamma$, being Frobenius, is isomorphic as a left $\Gamma$-module to $D ( \Gamma )$, and so adjointness gives
$$\Hom_{\Gamma}(B, \Gamma ) \simeq \Hom_{\Gamma}(B, D ( \Gamma ) ) \simeq D( \Gamma \otimes_{\Gamma} B ) \simeq D(B).$$
Therefore, the functor $\Hom_{\Gamma}(B,-)$ is isomorphic to $D(B) \otimes_{\Gamma}-$, proving the case $n=0$. For the case $n>0$, note that since $B$ is a projective $\La$-module, the $\La$-module $D(B) \otimes_{\Gamma} P$ is projective for every projective $\Gamma$-module $P$. However, since both $\La$ and $\Gamma$ are selfinjective, this means that $D(B) \otimes_{\Gamma} I$ is an injective $\La$-module for every injective $\Gamma$-module $I$. Moreover, given a $\Gamma$-module $Y$ with an injective resolution $\bf{I}$, the complex $D(B) \otimes_{\Gamma} \bf{I}$ is an injective resolution of the $\La$-module $D(B) \otimes_{\Gamma} Y$, since $D(B)$ is a projective right $\Gamma$-module. Therefore the isomorphisms
\begin{eqnarray*}
\Ext_{\Gamma}^n(B \otimes_{\La} X, Y) & \simeq & H^n \left ( \Hom_{\Gamma}(B \otimes_{\La} X, \bf{I} ) \right ) \\
& \simeq & H^n \left ( \Hom_{\La} (X, D(B) \otimes_{\Gamma} \bf{I} ) \right ) \\
& \simeq & \Ext_{\La}^n (X, D(B) \otimes_{\Gamma} Y)
\end{eqnarray*}
hold for every $\La$-module $X$. 
\end{proof}

In order to determine the support variety of $Au_{\lambda}$, we exploit some nice properties of certain bimodules arising from elements of the Hochschild cohomology ring. Let $\zeta$ be a homogeneous element of $\HH^* (A)$, represented by a map $\Omega_{A^{\e}}^{|\zeta|} (A) \xrightarrow{f_{\zeta}} A$, say. Then $\zeta$ corresponds to the bottom short exact sequence in the exact commutative pushout diagram
$$\xymatrix{0 \ar[r] & \Omega_{A^{\e}}^{|\zeta|} (A) \ar[d]^{f_{\zeta}} \ar[r] & P_{|\zeta|-1} \ar[d] \ar[r] & \Omega_{A^{\e}}^{|\zeta|-1} (A) \ar@{=}[d] \ar[r] & 0 \\
0 \ar[r] & A \ar[r] & K_{\zeta} \ar[r] & \Omega_{A^{\e}}^{|\zeta|-1} (A) \ar[r] & 0}$$
where $P_{|\zeta|-1}$ is the projective cover of $\Omega_{A^{\e}}^{|\zeta|-1} (A)$. If $\zeta$ is an element of $H$ and $M$ is an $A$-module, then by \cite[Proposition 4.3]{Erdmann} the support variety of the $A$-module $K_{\zeta} \otimes_A M$ is given by
$$\V_H (K_{\zeta} \otimes_A M) = \V_H ( \zeta ) \cap \V_H (M),$$
where $\V_H ( \zeta ) = \{ \alpha \in k^c \mid \zeta ( \alpha )=0 \}$. In the following result, we use the pushout bimodule $K_{\zeta}$ to give a criterion for when two lines in $k^c$ are perpendicular. Given a nonzero point $\mu \in k^c$, we denote the set $\{ \alpha \in k^c \mid \sum \alpha_i \mu_i =0 \}$ by $\ell_{\mu}^{\perp}$; this is the hyperplane perpendicular to the line $\ell_{\mu}$. 

\begin{lemma}\label{perpendicular}
Let $M$ be a periodic $A$-module of period $1$, and suppose $\V_H(M)$ is a single line $\ell_{\alpha}$, where $\alpha \in k^c$ is a nonzero point. Then, given any nonzero point $\mu \in k^c$, the implication
$$\underline{\Hom}_A(M, K_{\zeta} \otimes_A k) \neq 0 \Rightarrow \ell_{\alpha} \subseteq \ell_{\mu}^{\perp}$$
holds, where $\zeta = \sum \mu_i \eta_i \in H^2$.
\end{lemma}

\begin{proof}
Since the bimodule $K_{\zeta}$ is projective both as a left and as a right $A$-module, we see from the previous lemma that $\underline{\Hom}_A(M, K_{\zeta} \otimes_A k)$ is naturally isomorphic to $\underline{\Hom}_A(D(K_{\zeta}) \otimes_A M, k)$. Thus $\underline{\Hom}_A(M, K_{\zeta} \otimes_A k)$ is nonzero if and only if the same holds for $\underline{\Hom}_A(D(K_{\zeta}) \otimes_A M, k)$, and this is equivalent to $D(K_{\zeta}) \otimes_A M$ not being a projective $A$-module. The latter happens if and only if $\V_H (D(K_{\zeta}) \otimes_A M) \neq 0$.

Using the previous lemma once more, we see that the $H$-modules $\Ext_A^* (D(K_{\zeta}) \otimes_A M,k)$ and $\Ext_A^* (M, K_{\zeta} \otimes_A k)$ are isomorphic. Therefore
\begin{eqnarray*}
\V_H (D(K_{\zeta}) \otimes_A M) & = & \V_H (M, K_{\zeta} \otimes_A k) \\
& \subseteq & \V_H(M) \cap \V_H(K_{\zeta} \otimes_A k) \\
& = & \V_H(M) \cap \V_H ( \zeta ) \\
& = & \ell_{\alpha} \cap \ell_{\mu}^{\perp},
\end{eqnarray*}
and so if $\underline{\Hom}_A(M, K_{\zeta} \otimes_A k)$ is nonzero then $\ell_{\alpha} \subseteq \ell_{\mu}^{\perp}$.
\end{proof}

Next, define a map $k^c \xrightarrow{F} k^c$ of affine spaces by
$$( \alpha_1, \dots, \alpha_c ) \mapsto ( \alpha_1^a, \dots, \alpha_c^a ).$$
Our aim now is to show that $\V_H (Au_{\lambda}) = \ell_{F( \lambda )}$ for every nonzero point $\lambda \in k^c$. In order to prove this, we need the following lemma.

\begin{lemma}\label{monomorphism}
Let $\mu$ and $\lambda$ be nonzero points in $k^c$ with $\ell_{\mu} \subseteq \ell_{F( \lambda )}^{\perp}$, and denote the element $\sum \mu_i \eta_i$ in $H$ by $\zeta$. Then there exists a monomorphism 
$$Au_{\lambda} \to K_{\zeta} \otimes_A k,$$
and consequently $\underline{\Hom}_A(Au_{\lambda}, K_{\zeta} \otimes_A k)$ is nonzero.
\end{lemma}

\begin{proof}
Note that we have implicitly assumed that $c$ is at least $2$. Denote the radical of $A$ by $\ra$. The second syzygy $\Omega_A^2(k)$ is the kernel of the projective cover $A^c \to \ra$, which maps the $i$th generator $\varepsilon_i$ of the free $A$-module $A^c$ to $x_i$. The generators of $\Omega_A^2(k)$ are the two sets
$$\{ x_i^{a-1} \varepsilon_i \}_{i=1}^c, \hspace{3mm} \{ qx_j \varepsilon_i - x_i \varepsilon_j \}_{i<j},$$
and the element $z_i \in \Ext_A^2(k,k)$ is represented by the homomorphism $\Omega_A^2(k) \xrightarrow{f_{z_i}} k$ mapping $x_i^{a-1} \varepsilon_i$ to $1$ and all the other generators to zero. In the diagram
$$\xymatrix{
0 \ar[r] & \Omega_A^2(k) \ar[d]^{f_{z_i}} \ar[r] & A^c \ar@{..>}[d]^{f_ i} \ar[r] & \Omega_A^1(k) \ar@{..>}[d]^{\Omega_A^{-1}(f_{z_i})} \ar[r] & 0 \\
0 \ar[r] & k \ar[r]^{\cdot ( \prod x_i^{a-1})} & A \ar[r]^<<<<<{\pi} & \Omega_A^{-1}(k) \ar[r] & 0}$$
the map $f_i$, given by
$$e_j \mapsto \left \{
\begin{array}{ll}
q^{i-1} \prod_{n \neq i} x_n^{a-1} & \text{when } j=i \\
0 & \text{when } j \neq i,
\end{array}
\right.$$
provides the first step in a lifting of $f_{z_i}$. Thus the map $\Omega_A^{-1}(f_{z_i})$ factorizes as $\Omega_A^{-1}(f_{z_i}) = \pi g_i$, where $A \xrightarrow{\pi} \Omega_A^{-1}(k)$ is the natural quotient map and $\Omega_A^1(k) \xrightarrow{g_i} A$ is the map given by
$$x_j \mapsto \left \{
\begin{array}{ll}
q^{i-1} \prod_{n \neq i} x_n^{a-1} & \text{when } j=i \\
0 & \text{when } j \neq i.
\end{array}
\right.$$

Now we use the fact that $K_{\zeta} \otimes_A k$ is the pullback of
$$\xymatrix{
& \Omega_A^1(k) \ar[d]^{\Omega_A^{-1}( \sum \mu_i f_{z_i})} \\
A \ar[r]^<<<<<{\pi} & \Omega_A^{-1}(k)}$$
to obtain a monomorphism $Au_{\lambda} \to K_{\zeta} \otimes_A k$. We do this by constructing a monomorphism $Au_{\lambda} \xrightarrow{f} A \oplus \Omega_A^1(k)$ whose composition with $A \oplus \Omega_A^1(k) \xrightarrow{( \pi, \Omega_A^{-1}( \sum \mu_i f_{z_i}))} \Omega_A^{-1}(k)$ is zero. Namely, define $f$ by
$$u_{\lambda} \mapsto ( \sum_{i=1}^c \mu_i g_i ( u_{\lambda} ), -u_{\lambda} ).$$
To show that this map is well defined, we need to check that $f ( u_{\lambda}^a )=0$, that is, that the element $u_{\lambda}^{a-1} f ( u_{\lambda} )$ is zero. Computing directly, we see that
\begin{eqnarray*}
u_{\lambda}^{a-1} \sum_{i=1}^c \mu_i g_i ( u_{\lambda} ) & = & u_{\lambda}^{a-1} \sum_{i=1}^c \left ( \mu_i \lambda_i q^{i-1} \prod_{n \neq i} x_n^{a-1} \right ) \\
& = & \sum_{i=1}^c \lambda_i^{a-1} x_i^{a-1} \left ( \mu_i \lambda_i q^{i-1} \prod_{n \neq i} x_n^{a-1} \right ) \\
& = & \left ( \sum_{i=1}^c \mu_i \lambda_i^a \right ) \prod_{j=1}^c x_j^{a-1} \\
& = & 0
\end{eqnarray*}
since $\ell_{\mu} \subseteq \ell_{F( \lambda )}^{\perp}$. Thus the map $f$ is well defined, and it is obviously a monomorphism.
\end{proof}

We are now ready to prove that $\V_H (Au_{\lambda}) = \ell_{F( \lambda )}$ for every nonzero point $\lambda \in k^c$.

\begin{proposition}\label{periodicvariety}
If $\lambda \in k^c$ is nonzero, then $\V_H (Au_{\lambda}) = \ell_{F( \lambda )}$.
\end{proposition}

\begin{proof}
Denote by $T$ the $A$-module $Au_{\lambda} \oplus Au_{\lambda}^{a-1}$. This is a periodic module of period $1$, and its support variety equals that of $Au_{\lambda}$ since $Au_{\lambda}$ and $Au_{\lambda}^{a-1}$ are syzygies of each other. Moreover, the module $Au_{\lambda}$ is indecomposable, hence its support variety is a single line. Let therefore $\alpha \in k^c$ be a nonzero point such that $\V_H (Au_{\lambda})= \ell_{\alpha}$. Furthermore, let $\mu \in k^c$ be a nonzero point such that $\ell_{\mu} \subseteq \ell_{F( \lambda )}^{\perp}$, and denote the element $\mu_1 \eta_1 + \cdots + \mu_c \eta_c \in H$ by $\zeta$. By the previous lemma, the space $\underline{\Hom}_A(T, K_{\zeta} \otimes_A k)$ is nonzero, and so from Lemma \ref{perpendicular} we obtain the inclusion $\ell_{\alpha} \subseteq \ell_{\mu}^{\perp}$. Therefore $\ell_{\alpha} = \ell_{F( \lambda )}$.
\end{proof}

Using Proposition \ref{stablemap} and Proposition \ref{periodicvariety}, we now prove our main result. It shows that the map $k^c \xrightarrow{F} k^c$ maps the rank variety of a module onto its support variety. 

\begin{theorem}\label{mainthm}
$F \left ( \V_A^r(M) \right ) = \V_H (M)$ for every $A$-module $M$.
\end{theorem}

\begin{proof}
The result is clearly true if $M$ is projective. If $M$ is not projective, let $\lambda \in k^c$ be a nonzero point whose corresponding line $\ell_{\lambda}$ is contained in $\V_A^r(M)$. Then $\underline{\Hom}_A (Au_{\lambda},M)$ and $\underline{\Hom}_A (Au_{\lambda}^{a-1},M)$ are both nonzero by Proposition \ref{stablemap}, and so $\Ext_A^i(Au_{\lambda},M) \neq 0$ for $i>0$. This implies that $\V_H (Au_{\lambda},M)$ is nontrivial, and since this variety is contained in $\V_H (Au_{\lambda}) \cap \V_H(M)$, we see from  Proposition \ref{periodicvariety} that the line $\ell_{F( \lambda )}$ must be contained in $\V_H(M)$. Consequently, the inclusion $F \left ( \V_A^r(M) \right ) \subseteq \V_H (M)$ holds.

To prove the reverse inclusion, we argue by induction on the dimension of the support variety of $M$. We may assume that $M$ is indecomposable, since the variety (rank or support) of a direct sum is the union of the varieties of the summands. Suppose first that the dimension of $\V_H(M)$ is one, i.e.\ that $M$ is periodic. Then $\V_H(M)$ is a line, and we proved above that this variety contains $F \left ( \V_A^r(M) \right )$. If $F \left ( \V_A^r(M) \right )=0$, then $\V_A^r(M) =0$, which is not the case since $M$ is not projective. Hence $F \left ( \V_A^r(M) \right )$ is nontrivial, and so $F \left ( \V_A^r(M) \right ) = \V_H (M)$ in this case. Next, suppose that $\dim \V_H (M) >1$, and let $\mu \in k^c$ be a nonzero point such that $\dim \left ( \ell_{\mu}^{\perp} \cap \V_H (M) \right ) < \dim \V_H(M)$. Denote the element $\mu_1 \eta_1 + \cdots + \mu_c \eta_c \in H$ by $\zeta$, and let $\Omega_{A^{\e}}^2(A) \xrightarrow{f_{\zeta}} A$ be a bimodule map representing this element. Since the bottom exact sequence in the pushout diagram
$$\xymatrix{0 \ar[r] & \Omega_{A^{\e}}^2 (A) \ar[d]^{f_{\zeta}} \ar[r] & P_1 \ar[d] \ar[r] & \Omega_{A^{\e}}^1 (A) \ar@{=}[d] \ar[r] & 0 \\
0 \ar[r] & A \ar[r] & K_{\zeta} \ar[r] & \Omega_{A^{\e}}^1 (A) \ar[r] & 0}$$
splits as right $A$-modules, it stays exact after tensoring with $M$. This gives a short exact sequence
$$0 \to M \to K_{\zeta} \otimes_A M \to \Omega_A^1 (M) \oplus P \to 0,$$
in which $P$ is projective, and so 
$$\V_A^r (K_{\zeta} \otimes_A M) \subseteq \V_A^r(M) \cup \V_A^r ( \Omega_A^1 (M) \oplus P ) = \V_A^r(M).$$
Now since $\V_H (K_{\zeta} \otimes_A M) = \V_H ( \zeta ) \cap \V_H(M)= \ell_{\mu}^{\perp} \cap \V_H (M)$, induction gives $V_H(K_{\zeta} \otimes_A M) \subseteq F \left ( V_A^r(K_{\zeta} \otimes_A M) \right )$. 
Therefore $\ell_{\mu}^{\perp} \cap \V_H (M) \subseteq F( \V_A^r(M))$, and since this inclusion holds for every nonzero point $\mu \in k^c$ such that $\dim \left ( \ell_{\mu}^{\perp} \cap \V_H (M) \right )$ is strictly less than $\dim \V_H(M)$, we see that $\V_H(M) \subseteq F( \V_A^r(M))$.
\end{proof}

The Avrunin-Scott theorem follows immediately from this theorem; the rank variety of an $A$-module is isomorphic to its support variety.

\begin{corollary}\label{AvruninScott}
For every $A$-module $M$, the varieties $\V_A^r(M)$ and $\V_{\HH^{2*}(A)}(M)$ are isomorphic.
\end{corollary}

Consequently, the dimension of the rank variety of an $A$-module is the complexity of the module. In particular, an indecomposable module is periodic if and only if its rank variety is of dimension one.

\begin{corollary}\label{dimrankvar}
For every $A$-module $M$, the dimension of $\V_A^r(M)$ is the complexity of $M$.
\end{corollary}

\begin{corollary}\label{periodic}
An indecomposable $A$-module $M$ is periodic if and only if the dimension of $\V_A^r(M)$ is one.
\end{corollary}


\begin{thebibliography}{EHSST}
\bibitem[AGP]{Avramov}L.\ Avramov, V.\ Gasharov, I.\ Peeva,
  \emph{Complete intersection dimension}, Publ.\ Math.\ I.H.E.S.\ 86
  (1997), 67-114.
\bibitem[AvS]{AvruninScott}G.S.\ Avrunin, L.L.\ Scott, \emph{Quillen stratification for modules}, Invent.\ Math.\ 66 (1982), 277-286.
\bibitem[BEH]{Benson}D.\ Benson, K.\ Erdmann, M.\ Holloway, \emph{Rank
    varieties for a class of finite-dimensional local algebras}, J.\
  Pure Appl.\ Algebra 211 (2007), no.\ 2, 497-510.
\bibitem[BeO]{BerghOppermann}P.\ A.\ Bergh, S.\ Oppermann, \emph{Cohomology of twisted tensor products}, J.\ Algebra 320 (2008), no.\ 8, 3327-3338.
\bibitem[BeS]{BerghSolberg}P.\ A.\ Bergh, {\O}.\ Solberg, \emph{Relative support varieties}, to appear in Q.\ J.\ Math.
\bibitem[Car]{Carlson}J.\ F.\ Carlson, \emph{The varieties and cohomology ring of a module}, J.\ Algebra 85 (1983), 104-143.
\bibitem[ErH]{ErdmannHolloway}K.\ Erdmann, M.\ Holloway, \emph{The Avrunin and Scott theorem and a truncated polynomial algebra}, J.\ Algebra 299 (2006), no.\ 1, 344-373.
\bibitem[EHSST]{Erdmann}K.\ Erdmann, M.\ Holloway, N.\ Snashall,
  {\O}.\ Solberg, R.\ Taillefer, \emph{Support varieties for
    selfinjective algebras}, K-theory 33 (2004), 67-87.
\bibitem[Man]{Manin}I.\ Manin, \emph{Some remarks on Koszul algebras
    and quantum groups}, Ann.\ Inst.\ Fourier (Grenoble) 37 (1987),
  191-205.
\bibitem[Opp]{Oppermann}S.\ Oppermann, \emph{Hochschild homology and cohomology of quantum complete intersections}, preprint.
\bibitem[SnS]{Snashall}N.\ Snashall, {\O}.\ Solberg, \emph{Support
    varieties and Hochschild cohomology rings}, Proc.\ London Math.\
  Soc.\ 88 (2004), 705-732.
\end{thebibliography}
\end{document}